\documentclass[10pt]{article}
\usepackage{amssymb,amscd}
\usepackage{amsmath, amsthm}
\usepackage{graphicx}
\usepackage{psfrag}
\usepackage{comment}
\usepackage{amsrefs}

\begin{document}

\newtheorem{lemma}{Lemma}[section]
\newtheorem{corollary}{Corollary}[section]
\newtheorem{theorem}{Theorem}[section]
\newtheorem{statement}{Statement}[section]
\newtheorem{prop}{Proposition}[section]
\newtheorem{definition}{Definition}[section]
\newtheorem{note}{Note}[section]
\newtheorem{conjecture}{Conjecture}[section]

\newcommand {\C}{{\mathbb C}}
\newcommand {\D}{{\mathbb D}}
\newcommand{\De}{{\Delta}}
\newcommand{\de}{{\delta}}
\newcommand {\ra} {\rightarrow}
\newcommand{\la}{{\lambda}}
\newcommand{\La}{{\Lambda}}
\newcommand{\si}{{\sigma}}
\newcommand{\Si}{{\Sigma}}
\newcommand{\Om}{{\Omega}}
\newcommand{\om}{{\omega}}

\newcommand{\tl}{\tilde}
\newcommand{\sm}{\setminus}
\newcommand{\inter}{\operatorname{int}}
\newcommand{\di}{\partial}

\newcommand{\AAA}{{\cal{A}}}
\newcommand{\FF}{{\cal {F}}}
\newcommand{\MM}{{\cal{M}}}
\newcommand{\TT}{{\cal{T}}}
\newcommand{\QC}{{\cal {QC}}}
\newcommand{\SSS}{{\cal{S}}}

\newcommand{\N}{{\Bbb{N}}}

\newcommand{\Ups}{{\Upsilon}}

\renewcommand{\Im}{\operatorname{Im}}
\newcommand{\diam}{\operatorname{diam}}

\newcommand{\comm}[1]{}

\title{$\lambda$-Lemma for families of Riemann surfaces \\
  and the critical loci of complex  H\'enon maps}
\author{Tanya Firsova and Mikhail Lyubich}
\maketitle

\begin{abstract} We prove a version of the classical $\lambda$-lemma for
 holomorphic families of Riemann surfaces.
We then use it  to show that critical loci for complex H\'{e}non
maps that are small perturbations of quadratic polynomials with
Cantor Julia sets are all quasiconformally equivalent.
\end{abstract}

%\begin{abstract} We consider a family of disks with $n$ holes in $\mathbb C^2$.
%We assume that the boundary moves holomorphically in $\mathbb C^2$ and moreover that the holomorphic motion
%is smooth. We show that there exists an extension of the holomorphic motion to the interior.
%We use the extension to show that critical loci for complex H\'{e}non
%maps that are small perturbations of quadratic polynomials with
%disconnected Julia set are quasiconformally equivalent.
%\end{abstract}

\section{Introduction}

A {\it holomorphic motion} in dimension one  is a family of
injections $h_\lambda:  A \rightarrow \hat{\mathbb C}$ of some set
$A\subset \hat{\mathbb C}$ holomorphically depending on a
parameter $\lambda$ (ranging over some complex manifold
$\Lambda$). It turned out to be  one of the most useful tools in
one-dimensional complex dynamics. First it was used to prove that
a generic rational endomorphism $f: \bar{\mathbb C}\rightarrow
\bar{\mathbb C}$ is structurally stable (see \cite{MSS,L}), and
then has found numerous further applications.

%The notion of holomorphic motions was introduced my Ma\~{n}\'{e}, Sad and Sullivan \cite{MSS} in their seminal
%paper ``On the dynamics of rational maps''. Many important dynamical objects for complex
%dynamical systems move holomorphically. In particular, the Julia set of a rational function moves holomorphically inside a hyperbolic component.
%And hence, they are quasiconformally equivalent to each other by
%Ma\~{n}\'{e}, Sad and Sullivan's Lambda Lemma.

Usefulness of holomorphic motions largely comes from their nice
extension and regularity properties usually referred to as the
{\it $\lambda$-lemma}. The simplest version of the Extension
$\lambda$-lemma asserts that the holomorphic motion of any subset
$X\subset \hat{\mathbb C}$ extends to a holomorphic motion of the
closure $\bar X$ \cite{MSS,L}. A more advanced version says that
it  extends to the whole Riemann sphere  over a smaller parameter
domain \cite{BR,ST}. The strongest version asserts that if
$\Lambda$ is the disk $\mathbb D\subset \mathbb C$ then the
extension is globally defined, over the whole $\mathbb D$.
Moreover, the maps $h_\lambda$  are automatically continuous
\cite{MSS,L} and in fact, quasiconformal \cite{MSS}.

% Dynamically significant objects in $\mathbb C^2$ often move holomorphically.
In dimension two, holomorphic motions $h_\la : A \ra \C^2$,
$A\subset \C^2$, do not have such nice properties: in general,
they do not admit extension even to the closure $\bar A$, and the
maps $h_\la$ are not automatically continuous (let alone,
quasiconformal).
%The dependence on a point in $\mathbb C^2$ does not need to be quasiconformal or continuous,
%they do not need to extend uniquely to the closure of the set.
%It is interesting to see whether dynamical holomorphic motions, have good properties, analogous to one dimensional case.
Still, under some circumstances, holomorphic motions turn out to be  useful   in higher dimensions as well,
% This question was in a particular case studied in
see \cite{BB,DL}.

In this paper, we prove a version of  the $\la$-lemma for
a class of  holomorphic motions in $\mathbb C^2$ that naturally arise in the study of complex H\'{e}non maps.
Namely, we consider a holomorphic
family of Riemann surfaces $S_{\lambda}\subset \mathbb C^2$  that fit into a complex two-dimensional manifold
such that the boundaries of $S_{\lambda}$ move holomorphically in $\mathbb C^2$.
We show that under suitable conditions, the holomorphic motion of the boundary can be extended to a holomorphic motion of the
surfaces. The proof  is based upon Teichm\"{u}ller Theory.

This work is motivated by study of the geometry of the critical locus $\mathcal C$ for the H\'enon
automorphisms
$$
    f: (x,y) \mapsto (x^2+c-ay, x)
$$
of ${\mathbb C}^2$. This locus was introduced by Hubbard (see \cite{BS5}) as the set of tangencies between two dynamically defined foliations outside the ``big'' Julia set.
It was studied in \cite{LR,F} in the case of small perturbations (i.e., with a small Jacobian $a$)
of one-dimensional hyperbolic polynomials
$P_c: x\mapsto x^2+c$. In case when $c$ is outside the Mandelbrot set (and $a$ is small enough), the critical locus has a rich topology
described in \cite{F}. Our version of the $\la$-lemma implies that all these critical loci are
quasiconformally equivalent.

\section{Background}
\subsection{Notations}
We will use the following notations throughout the paper:
$\Delta$ for the unit disk, $\mathbb H$ for the hyperbolic plane, $\hat{\mathbb C}$ for the Riemann sphere.

\subsection{$\la$-lemma}
Let $M$ be a complex manifold, and let $\Delta\subset \mathbb C$ be a unit disk.

\begin{definition} Let $A\subset M$. A holomorphic motion of  $A$ over $\De$ is a map $f:\Delta\times A \to M$ such that:
\begin{enumerate}
\item For any $a\in A$, the map $\lambda\mapsto f(\lambda,a)$ is holomorphic in $\Delta$;
\item For any $\lambda\in \Delta$, the map $a\mapsto f(\lambda,a)=:f_{\lambda}(a)$ is an injection;
\item The map $f_0$ is the identity on $A$.
\end{enumerate}
\end{definition}

Holomorphic motions in one-dimensional dynamical context first appeared in \cite{L,MSS}.
The following simple but important virtues of one-dimensional  holomorphic motions
are usually referred to as $\lambda$-lemma:

\newtheorem*{Extension lemma}{Extension $\la$-lemma}
\begin{Extension lemma}[\cite{L,MSS}]
Let $M=\hat \C$,  $A\subset \hat \C$.
 Any holomorphic motion $f:\Delta\times A \to \hat\C$
     extends to a holomorphic motion  $\De\times \bar A \ra \hat \C$.
\end{Extension lemma}

\begin{definition} Let $(X,d_X),$ $(Y,d_Y)$ be two metric spaces. A homeomorphism $f:X\to Y$ is said to be $\eta$-quasisymmetric, if there exists an increasing continuous function $\eta:[0,\infty)\to [0,\infty)$, such that for any triple of distinct points $x,$ $y$ and $z:$
$$
\frac{d_Y(f(x), f(y))}{d_Y(f(x), f(z))}\leq \eta\left(\frac{d_X(x,y)}{d_X(x,z)}\right)
$$
\end{definition}

A quasisymmetric map between two open domains is quasiconformal.

\newtheorem*{Qc lemma}{Qc $\la$-lemma}
\begin{Qc lemma}[\cite{MSS}]
Under the circumstances of the Extension $\la$-lemma,
for any $\lambda\in \De,$ the map $f_{\lambda}: \bar A \ra \bar A$ is quasisymmetric.
\end{Qc lemma}

Later,  Bers \& Royden \cite{BR} and
Sullivan \& Thurston \cite{ST} proved
that there exists a universal $\delta>0$
such that under the circumstances of the Extension $\la$-lemma,
the restriction of $f$ to the parameter disk $\D_\de$ of radius $\delta$ can be extended to
a holomorphis motion  $\De_\de \times \hat {{\mathbb C}} \ra \hat \C$ (``BRST $\la$-lemma'').
Though this version of the $\lambda$-lemma will be sufficient for our dynamical applications,
let us also state the strongest version  asserting
that $\delta$ is actually equal to $1$:
%Finally, Slodkowski \cite{Slodkowski} proved that $\delta$ is actually equal to $1$:

\newtheorem*{Slod lemma}{Slodkowski's $\la$-lemma}
\begin{Slod lemma}
Let  $A\subset \hat \C$.
Any holomorphic motion $f:\Delta\times A \to \hat\C$
extends to a holomorphic motion  $\De\times \hat{\mathbb C} \ra \hat\C$.
\end{Slod lemma}

In what follows, we will use the same notation $f$ for the extended holomorphic motion.

\subsection{Elements of Teichm\"uller Theory}
We assume that the reader is familiar with the basics of Teichm\"{u}ller Theory.
To set up terminology and notation, we recall some basic definitions and statements and refer to \cite{Teich} for details.

Given a base Riemann surface $S$,
let $\QC(S)$ stand for the set of all Riemann surfaces quasiconformally equivalent to $S$.

\begin{definition} Let $X_1, X_2\in \QC(S) $, and let  $\phi_i :S\to X_i$ be quasiconformal mappings.
The pairs $(X_1,\phi_1)$ and $(X_2,\phi_2)$ are called {\it Teichm\"{u}ller equivalent}
if there exists a conformal isomorphism $\alpha:X_1\to X_2$ such that
% $\phi_2=\alpha\circ\phi_1$ on the ideal boundary $I(S)$ and
$\phi_2$ is homotopic to $\alpha\circ\phi_1$ relative to the ideal boundary $I(S)$.
The class of equivalent pairs is called a marked by $S$ Riemann surface.%
\footnote{Somewhat informally, we will use notation $(X, \phi)$, or just $X$, for the equivalence class.}
\end{definition}

\begin{definition} The Teichm\"{u}ller space ${\cal T}(S)$ modeled on $S$
is the space of marked by $S$ Riemann surfaces.
\end{definition}

The space $\TT(S)$  can be endowed with a natural  {\it Teichm\"{u}ller metric}.

Any marked Riemann surface  $(\tilde{S}, \psi ) \in {\cal T}(S)$
defines an isometry
\begin{equation}\label{change}
\psi_*: {\cal T}(S) \ra  {\cal T}(\tilde{S}), \quad
  \psi_*:(X, \phi)\to (X, \phi \circ \psi^{-1}),
\end{equation}
 called a {\it change of the base point} of the Teichm\"{u}ller space.

\begin{definition} {\it A Beltrami form} $\mu$ on $S$ is a measurable $(-1,1)$-differential form
with $|\mu(z)|< 1$ a.e. It is called bounded if $\| \mu\|_\infty < 1$.
\end{definition}

Locally, $\mu$ can be represented as
$\displaystyle{\mu(z)\frac{d\bar{z}}{dz}}$, where $\mu(z)$ is a measurable function with  $|\mu(z) |<1$ a.e.
(Notice that the latter condition is independent of the choice of the local coordinate.)

Any Beltrami form $\mu$ determines a {\it conformal structure} on $S$,
i.e., the class of metrics conformally equivalent to $ dz + \mu(z)\,  d\bar z$.
(In what follows, Beltrami forms and the corresponding conformal structures will be freely identified.)
The standard structure $\si$ corresponds to $\mu\equiv 0$.

Let $\MM(S)$ be the space of bounded Beltrami forms on $S$.
It is identified with the unit ball in the complex Banach space $L^\infty(S)$,
from which it inherits a natural complex structure.

Any quasiconformal map $f: S\ra X$ induces the pullback
\begin{equation}\label{eq:pull_back}
 f^*: \MM(X) \ra \MM(S).
\end{equation}

\newtheorem*{MRMT}{Measurable Riemann Mapping Theorem}
\begin{MRMT} Let $\mu$ be a bounded Beltrami form on $S$ with
$\|\mu\|_\infty = k <1$. Then there exists a Riemann surface
$S_\mu \in \QC(S)$ and  a $K$-quasiconformal map $f_\mu : S\to
S_\mu  $ with $K=(1+k)/(1-k)$ such  that $f_\mu^*\sigma=\mu$.
Moreover, it is unique up to postcomposition with some conformal
map $h: S_\mu \to S_\mu' $.
\end{MRMT}

Analytically, $f= f_\mu$ gives a solution to the {\it Beltrami equation}
\begin{equation}\label{Beltrami eq}
 \frac{\partial f }{\partial \bar{z}}= \mu \frac{\partial f}{\partial z}.
\end{equation}

By the Measurable Riemann Mapping Theorem, there is a natural projection
$$\Phi_S: \MM(S)\to {\cal T}(S).$$

The pullback operator from equation (\ref{eq:pull_back}) descends
to $f^*:\TT(X)\to \TT(S)$. It is the inverse of the change of the
base point $f_*$.

\begin{theorem}
There exists a unique complex structure on ${\cal T}(S)$ such that the projection  $\Phi_S$
is holomorphic.
\end{theorem}

Notice that the change of the base point (\ref{change}) is a biholomorphism $\TT(S)\ra \TT(\tl S)$,
so the complex structure on the Teichm\"uller space  is independent of the choice of $S$.

\begin{prop}[Slodkowski's $\lambda$-lemma restated \cite{Teich}]
Every holomorphic map $\gamma:\Delta\to {\cal T}(S)$
lifts to a holomorphic map $\tilde{\gamma}:\Delta\to {\cal \MM}(S)$.
% such that $\Phi_X\circ {\tilde \gamma}=\gamma$.
\end{prop}

\medskip
Let $S$ be a hyperbolic  Riemann surface,
and let $p:\mathbb H\to S$
be its universal covering with the group of deck transformations $\Gamma$.

\begin{lemma}[\cite{Teich}] Let $\nu$ be an infinitesimal Beltrami form on $S$, then $\nu\in \mbox{Ker}\ d \Phi_{S}$
if and only if $p^*\nu=\bar{\partial} \eta$, where $\eta$ is a continuous $\Gamma$-invariant vector field on
$\bar{\mathbb H}$ such that the distributional derivative $\bar{\partial} \eta$ has bounded $L^{\infty}$-norm and $\eta=0$ on
$\overline{\mathbb R}$.
\end{lemma}

\begin{corollary}\label{cor:Ker} Assume that $S$ is a bounded type Riemann surface with the boundary
$\partial S=\gamma^1\cup \dots\gamma^n$, where $\gamma^i$ are
smooth Jordan curves. Let $\nu$ be an infinitesimal Beltrami form.
Then $\nu\in \mbox{Ker}\ d\Phi_S$ if and only if
$\nu=\bar{\partial} \xi$, where $\xi$ a continuous vector field on
$S$ such that the distributional derivative $\bar{\partial} \xi$
has bounded $L^{\infty}$ norm and $\xi=0$ on $\partial S$.
\end{corollary}

\begin{proof} Let $p^{-1}(\xi)$ be a lift of the vector field $\xi$ to $
\mathbb H$. Let $D$ be a fundamental domain of the group $\Gamma$.
The vector field $\xi$ vanishes on the boundary. Therefore, $p^{-1}(\xi)|_D$ is bounded in the hyperbolic metric.
Since M\"{o}bius transformations preserve the the hyperbolic metric, $p^{-1}(\xi)$ is bounded in hyperbolic metric on $\mathbb H$.
Thus, it vanishes on the boundary in the Euclidean metric.
\end{proof}

The group $\Gamma$ is Fuchsian, so it acts on the whole Riemann
sphere $\hat{\mathbb C} $. Let $\MM^{\Gamma}(\hat {\mathbb C}
)\subset \MM(\hat {\mathbb C} )$  be the space  of
$\Gamma$-invariant Beltrami forms on $\hat{\mathbb C}$. We can map
$\MM(S)$ to $\MM^{\Gamma}(\mathbb C)$ by lifting $\mu\in \MM(S)$
to the Beltrami form $\hat{\mu}= p^*\mu$ on $\mathbb H$ and then
extending it by  $0$ to the rest of $\hat {\mathbb C}$. By the
Measurable Riemann Mapping Theorem, there exists a unique solution
$f_{\hat{\mu}}: \hat \C\ra \hat \C$ of Beltrami equation
(\ref{Beltrami eq}) for $\hat \mu$,
% $$
% \frac{\partial g^{\hat{\mu}}}{\partial \bar{z}}=\hat{\mu}\frac{\partial g^{\hat{\mu}}}{\partial z}
%$$
fixing $0,1$ and $\infty$.
%normalized so  that $g^{\hat{\mu}}(0)=0,$ $g^{\hat{\mu}}(1)=1, $  and $g^{\hat{\mu}}(\infty)=\infty $.
It conjugates the Fuchsian group $\Gamma$ to a quasi-Fuchsian group $\Gamma_\mu$ preserving the quasidisk
$f_\mu(\mathbb H)$. Hence it induces a quasiconformal map $S\ra S_\mu$ (for which we will keep the same notation
$f_\mu$).

Consider the map
$$
   \Psi : \MM (S)\times \hat{\mathbb C} \to \MM (S)\times \hat{\mathbb C},\quad
   (\mu,z) \mapsto (\mu, f_{\hat{\mu}}(z)).
$$
The image $\Psi(\MM(S)\times \mathbb H)$ is an open subset of $ \MM(S)\times \hat{\mathbb C} $
called the {\it Bers fiber space}.  Fiberwise actions of quasi-Fuchsian groups $\Gamma_\mu$
induce an action of $\Gamma$ on the Bers fiber space.

\begin{definition}
The quotient $\Psi(\MM(S)\times \mathbb H) /\Gamma$
is called  the Universal Curve over $\MM(S)$.
\end{definition}

\section{$\lambda$-Lemma for families of Riemann surfaces}
%We consider $\mathbb C^2$ with coordinates $z=(z_1,z_2)$, % and we  let $z=(z_1,z_2)$.

Let us consider a complex 3-fold $\De\times \C^2$,
and let $\pi_1: \De \times \C^2 \ra \De$ be the natural projection to $\De$.
Let $\bar \SSS\subset \De \times \C^2$ be a complex 2-fold with boundary such that
$\pi_1: \bar\SSS \ra \De$ is a smooth locally trivial fibration with fibers
$\bar S_\la$. We assume that the fibers  $\bar S_\la$ are compact Riemann surfaces with boundary
$\di S_\la= \gamma_\la^1\cup\dots \cup \gamma_\la^n$, where the $\gamma_\la^i$ are smooth Jordan curves
that move holomorphically over $\De$. Intrinsic interior of $\bar \SSS$ is a complex 2-fold
$\SSS= \bar \SSS \sm \di \SSS$ that fibers over $\De$.  The fibers are open Riemann surfaces
$$
   S_\la= \operatorname{int} \bar S_\la= \bar S_\la \setminus \di \bar S_\la.
$$
Note that since $\De$ is contractible,
the fibration $\pi_1 : {\cal S}\ra \De $ is globally trivial in the smooth category.
% diffeomorphic, over $\De$, to the product $\Delta\times {\bar S}_0$.

\comm{****
Let $\gamma^1_{\lambda},\dots,\gamma^n_{\lambda}$ be a collection of curves that move holomorphically
in $\mathbb C^2$. We assume that for each
$\lambda$, the     $\gamma^i_{\lambda}$
are boundary curves of a compact Riemann surface $\bar S_{\lambda}\subset \C^2$,
and we let
$$
   S_\la= \operatorname{int} \bar S_\la= \bar S_\la \setminus \bigcup_i \gamma_\la^i.
$$
We also assume that surfaces $\bar S_{\lambda}$ form a fibration over $\De$ whose total space is
a complex manifold $\bar \SSS=\bigcup_\la \bar S_\la$ with boundary%
\footnote{Notice that $\bar \SSS$ is not compact}
 $$
\di \bar \SSS_\la = \bigcup_\la \bigcup_i g_\la^i
$$
in $\Delta\times \mathbb C^2.$
We let
$$
\SSS_\la= \inter \bar \SSS_\la= \bigcup_{\la\in \De} S_\la.
$$

Below we state precisely the conditions that we impose.

\begin{enumerate}
\item Let ${\cal S}\subset
\Delta\times\mathbb C^2$ is a $2$-dimensional complex manifold with a smooth boundary. Let $\pi_1:{\cal S}\to \Delta$ be the projection to the first coordinate. Let
$$S_{\lambda}:=\pi_1^{-1}(\lambda)\cap \mbox{int}{\cal S}.$$
\item The surfaces $S_\lambda$ are diffeomorphic to disks with $n$ holes. Let $\bar{S}_{\lambda}$ be the topological boundary of $S_{\lambda}$ in $\mathbb C^2$. We assume that $\bar{S}_{\lambda}=\pi_1^{-1}(\lambda)$. Moreover, we assume ${\cal S}$ is diffeomorphic to a product $\Delta\times {\bar S}_0$. %Let $\pi_2:{\cal S}\to
%bar{S}_0$ be the projection. Let $\gamma_0^1$,\dots, $\gamma_0^n$ be the %boundary components of $S_0$. Let
%$$\gamma_{\lambda}^n=\pi_2^{-1}(\gamma_0^n)\cap \bar{S}_{\lambda}.$$
\end{enumerate}

We prove the following theorem:
****}

\begin{theorem}\label{main}
 Let $f:\Delta \times \partial S_0 \to \mathbb C^2$ be a holomorphic motion of $\partial S_0$ over $\De$,
and let  $f_\lambda(z)=f(\lambda,z)$,  $\Im f_\lambda = \partial S_{\lambda}$.
Moreover, assume that the maps $f_{\lambda}: \partial S_0\ra \partial S_\la$ are diffeomorphisms.
Then there exists a holomorphic motion $\tilde f$ of $\bar{S}_0$ over $\Delta$, such that
\begin{enumerate}

\item $f=\tilde{f}|_{\partial S_0}$;
\item for any $\lambda\in \De $, $\Im \tilde{f}_\la = S_{\lambda}$.
\end{enumerate}
\end{theorem}

We will show that a family $S_{\lambda}$ can be realized as a
holomorphic curve in the Universal Curve over the Teichm\"{u}ller
space ${\cal T}(S_0)$.

%This topological trivialization defines a preferred homotopy class of maps
%$h_{\lambda}:(S_0, \partial S_0)\to (S_{\lambda}, \partial S_{\lambda})$,
%so that $h_{\lambda}|_{\partial S_0}=f_{\lambda}$.

%Since maps $f_{\lambda}:\partial S_0\to \partial S_{\lambda}$ are smooth,
%they can be extended to a quasiconformal map $f_{\lambda}:S_0\to S_{\lambda}$.
%We choose a representative in the homotopy class defined by the fibration $\pi$.
%We consider a Teichm\"{u}ller space $T(S_0)$ modeled on $S_0$.
%We want to show that the map lifts to a quasisymmetric map between the %boundaries of the universal covers.
%In our case we have a preferred homotopy class of maps between our surfaces. %We can also normalize the uniformization map by saying that given three %points on the boundary go to their images under the holomorphic motion. Thus %the map lifts to the boundary maps between uniformizations. We want to show %that this map is quasisymmetric.

%\begin{lemma} Let $X$ to $Y$ be two homeomorphic surfaces. Let
%$f:I(X)\to I(Y)$ be quasysimmetric. Then it can be extended to a %quasiconformal map $f:X\to Y$, in a given homotopy class.
%\end{lemma}

%\begin{proof} Let $\phi:D\to X$ be a uniformization map.
%\end{proof}

Let us first extend the holomorphic motion $f$ to  a smooth motion of $S_{0}\ra S_\la$ over $\De$,
for which we will use the same notation $f_\la$ as for the original motion.
It defines a smooth curve $\tau_{\lambda}:=(S_{\lambda}, f_{\lambda})$ in the Teichm\"{u}ller space ${\cal T}(S_0)$.

\begin{lemma}
   The elements $\tau_{\lambda}\in {\cal T}(S_0)$ do not depend on the choice of extension.
\end{lemma}

\begin{proof} Let $f_{\lambda}$ and $g_{\lambda}$ be two  extensions as above. Then
$$
   g_{\lambda}^{-1}\circ f_{\lambda}:\bar{S}_0\to \bar{S}_0, \quad g_{\lambda}^{-1}\circ f_{\lambda}|_{\, \di S_0} = \mathrm{Id}, \quad \la \in \De.
$$
Hence the maps $g_{\lambda}^{-1}\circ f_{\lambda}$ are homotopic to identity rel $\di S_0$, and thus
define the same element of the Teichmu\"{u}ller space ${\cal T}(S_0)$.
\end{proof}

%In what follows, it will be sometimes convenient to change the
%base point from $S_0$ to $S_{\lambda}$.

\begin{lemma}\label{lem:form}
There exists a holomorphic $1$-form $\omega$ on $S_0$ that extends smoothly to the boundary and $\omega(z)\neq 0$
for all $z\in \bar{S}_0$.
\end{lemma}

\begin{proof} Let $R$ be a Shottky double cover of $S_0$ \cite{AS}. There is a holomorphic embedding $\phi:S_0\to R$ such that
$\phi$ extends smoothly to the boundary $\partial S_0$. By Riemann-Roch theorem, we can take a meromorphic form $u$ on $R$ such that zeroes
and poles of $u$ belong to $R\backslash \bar{S}_0$. The form $\omega=u|_{S_0}$ is a desired holomorphic $1$-form.
\end{proof}

\begin{theorem}
The curve $\tau_{\lambda}$ is an analytic curve in ${\cal T}(S_{0})$.
\end{theorem}

\begin{proof}
Let us show that $\displaystyle{ \frac {\partial\tau_\la }  { \partial \overline{\lambda}} =0} $.

Fix some $\lambda_0\in \De $. Consider the map $f_{\lambda}\circ f^{-1}_{\lambda_0}:S_{\lambda_0}\to S_{\lambda}$. This map
defines a family $\mu_{\lambda}$ of Beltrami forms on $S_{\lambda_0}$:

$$\mu_{\lambda}=\frac{\overline{\partial}\left(f_{\lambda}\circ f^{-1}_{\lambda_0}\right)}{\partial \left(f_{\lambda}
\circ f^{-1}_{\lambda_0}\right)}\in {\cal M}(S_{\lambda_0})$$

%let $\Sigma(D)$ be the universal curve over $M(D)$.
%Let $\pi:\Sigma(D)\to M(D)$ be the projection.
%Let ${\cal S}_1:=\pi^{-1}(\mu_{\lambda}), \forall \lambda\in \Delta$
%\begin{lemma}${\cal S}$ is biholomorphically equivalent to $\cal S$.
%\end{lemma}

Consider the projection map
$$\Phi_{\lambda_0}: \MM (S_{\lambda_0})\to {\cal T}(S_{\lambda_0})$$

The map $(f_{\lambda_0})^*$ provides an isomorphism between ${\cal T}(S_{\lambda_0})$ and
${\cal T}(S_0)$.

Moreover,
$$
  (f_{\lambda_0})^*\circ \Phi_{\lambda_0}:{\cal M} (S_{\lambda_0})\to {\cal T}(S_0),
$$
$$
  (f_{\lambda_0})^*\circ \Phi_{\lambda_0}(\mu_{\lambda})=\tau_{\lambda}.
$$
%We consider a projection of the family to the Teichmuller
%space of a disk and then show that the projection defines an
%analytic curve in the Teichmuller space.
Then we have:
$$
   \left.\frac{\partial \tau_\la }{\partial \bar{\lambda}} \right|_{\la= \lambda_0}
= d f_{\lambda_0}^{*}\circ d\Phi_{\lambda_0}
   \left.\frac{\partial\mu_\la }{\partial \bar{\lambda}}  \right|_{\la= \lambda_0}
$$

Let us show that $\displaystyle{\frac{\partial \mu_{\lambda}}{\partial \overline{\lambda}}(\lambda_0)\in
\mbox{Ker}\ d \Phi_{\lambda_0}}$.
To simplify the notations, we assume below $\lambda_0=0$.
We construct a vector field $\xi$ on $S_0$, such that
$\displaystyle{\frac{\partial \mu_{\lambda}}{\partial \overline{\lambda}}
(0)=\overline{\partial} \xi}$, and $\xi=0$ on $\partial S_0$ and apply Corollary \ref{cor:Ker}.
Let $\displaystyle{\nu:=\frac{\partial\mu_{\lambda}}{\partial \lambda}(0)}$,
$\displaystyle{\kappa:=\frac{\partial \mu_{\lambda}}{\partial \overline{\lambda}}(0)}$. Since $\mu_0=0$,
$$\mu_{\lambda}=\lambda \nu+\bar{\lambda}\kappa+o(\lambda, \bar{\lambda}).$$
Let $(g_1, g_2):S_0\to \mathbb C^2$ be the defining functions of the Riemann surface $S_0$.
The functions $g_1$, $g_2$ extend smoothly to the boundary, and
$$f_{\lambda}=\left(\begin{array}{l}g_1+\lambda
u_1+\bar{\lambda}v_1+o(\lambda,\bar{\lambda})\\ g_2+\lambda
u_2+\bar{\lambda}v_2+o(\lambda,\bar{\lambda})\end{array}\right)$$

Since $f_{\lambda}$ is a holomorphic motion on the boundary, functions $v_1$ and $v_2$ are
equal to zero on the boundary. Let $w$ be a local coordinate on $S_0$, $\partial f=\frac{\partial f}{\partial w}dw$,
$\bar{\partial} f=\frac{\partial f}{\partial \bar{w}}d\bar{w}$.
By Lemma \ref{lem:form} there is a holomorphic non-zero $1$-form $\omega$ on $S_0$ that extends smoothly to the boundary $\partial S_0$.

The functions $g_1$ and $g_2$ are holomorphic. Thus,
$\partial g_1=h_1\omega$, $\partial g_2=h_2\omega$, where $h_1$, $h_2$ are holomorphic functions on $S_0$ that
extend smoothly to $\partial S_0$.

$$\partial f_{\lambda}=\left(\begin{array}{l}h_1\omega+\lambda\partial u_1+\bar{\lambda}\partial v_1+\dots\\
h_2\omega+\lambda \partial u_2+\bar{\lambda}\partial v_2+\dots\end{array}\right)\quad\bar{\partial} f_{\lambda}=
\left(\begin{array}{l}\lambda \bar{\partial} u_1+\bar{\lambda}\bar{\partial} v_1+\dots\\\lambda
\bar{\partial} u_2+\bar{\lambda}\bar{\partial}
v_2+\dots\end{array}\right)$$
$$\mu_{\lambda}\partial f_{\lambda}=\bar{\partial} f_{\lambda}$$
$$\left(\lambda \nu +\bar{\lambda} \kappa+\dots\right)\left(\begin{array}{l}h_1\omega+\dots\\h_2\omega+\dots\end{array}\right)
=\left(\begin{array}{l}\lambda \bar{\partial}u_1+\bar{\lambda}\bar{\partial}v_1+\dots\\
\lambda\bar{\partial} u_2+\bar{\lambda}\bar{\partial} v_2+\dots\end{array}
\right)$$

Therefore, $\kappa \left(\begin{array}{l}h_1\omega\\h_2\omega\end{array}\right)=\left(\begin{array}{l}\bar{\partial} v_1\\
\bar{\partial}v_2\end{array}\right)$. It follows from
\cite{Voichick} that the space of maximal ideals in the algebra
$A$ of holomorphic functions on $S_0$ that extend continuously to
the boundary is isomorphic to $\bar{S}_0$. The functions $h_1$,
$h_2$ do not have common zeroes on $\bar{S}_0$. So the ideal
generated by $h_1$ and $h_2$ coincide with $A$, in particular
function $1$ belong to the ideal. Hence there exists a pair of
holomorphic functions $s_1$ and $s_2$ on $S_0$ that extend
continuously to $\partial S_0$ so that $s_1h_1+s_2h_2=1$. Let
$\eta$ be a holomorphic vector field on $S_0$, such that $\omega
(\eta)=1$. Since $\omega$ extends smoothly to $\partial S_0$,
$\eta$ extends smoothly to $\partial S_0$. Set
$$\xi=(s_1v_1+s_2v_2)\eta,$$
then $\kappa=\bar{\partial}\xi$. Functions $v_1$ and $v_2$ are
smooth in $\bar{S}_0$, so $\bar{\partial} v_1$ and $\bar{\partial}
v_2$ are bounded in $L^{\infty}$-norm. They are also equal to $0$
on the boundary of $S_0$, so by Corollary~\ref{cor:Ker} $\kappa\in
\mbox{Ker}\, d\Phi_{\lambda_0}$.
\end{proof}

\begin{proof}[Proof of Theorem \ref{main}:]

By Slodkowski's $\lambda$-lemma, there exists a holomorphic family $\nu_{\lambda}$ on $S_0$, so that
$\Phi_{\lambda_0}\nu_{\lambda}=\tau_{\lambda}$. Notice that ${\cal S}$ is the preimage of the
family $\{\nu_{\lambda}|\ \lambda\in \Delta\}$ in the Universal Curve over $\MM(S_0)$.

\end{proof}

\section{Application to dynamics}

\subsection{Background on H\'{e}non maps}
Complex H\'{e}non maps are biholomorphisms $f_\la: \C^2\ra \C^2$  of the form
$$
f_{\lambda}\left(\begin{array}{c}x\\y\end{array}\right)=\left(\begin{array}{c}
x^2+c-ay \\ x
\end{array}\right),
$$
where $\lambda=(a,c)$, $a\in \mathbb C^{*}$, $c\in \mathbb C$.
%If one takes $(x,y)\in \mathbb C^2$, then a H\'{e}non map is a biholomorphism of $\mathbb C^2$.

In the one-dimensional holomorphic dynamics, %the dynamics of the map
the global phase portrait is to a large extent determined by the
behavior of the critical points.
Being diffeomorphisms,
H\'{e}non maps do not have critical points in the usual sense.
However, they possess an interesting analogous object, the {\it critical locus}.
% for H\'{e}non maps,  of the critical points.

Let us recall the following dynamically significant sets:
$$
U_{\lambda}^+=\{(x,y):\ f_{\lambda}^n(x,y)\to \infty\ \mbox{as}\
n\to +\infty \},\quad K_{\lambda}^+=\mathbb C^2\backslash
U_{\lambda}^+,\quad J_{\lambda}^+=\partial K_{\lambda}^+,
$$
$$
U_{\lambda}^-=\{(x,y):\ f_{\lambda}^{-n}(x,y)\to \infty\
\mbox{as}\ n\to +\infty \},\quad K_{\lambda}^-=\mathbb
C^2\backslash U_{\lambda}^-,\quad J_{\lambda}^-=\partial
K_{\lambda}^- ,
$$
$$J_{\lambda}=J_{\lambda}^+\cap J_{\lambda}^-.$$

Domains $U_{\lambda}^+$ and $U_{\lambda}^-$ are called {\it
(forward and backward) escape loci}; $J_{\lambda}$ is called  the
{\it Julia set} of the H\'{e}non map.

In the one-dimensional polynomial dynamics, critical points of the polynomial
are  critical points of the Green's function on the complement of the filled Julia set.
For a complex H\'{e}non map, one can define the {\it forward and backward  Green's functions}
 that measure the escape rate of the orbits  under forward and backward iterations of
the map \cite{HOVI}:
$$G_{\lambda}^{+}(x,y)=\lim_{n\to\infty}\frac{\log^+|f^{n}_{\lambda}(x,y)|}{2^n}, $$

$$G_{\lambda}^-(x,y)=\lim_{n\to \infty}\frac{\log^+|f^{-n}_{\lambda}(x,y)|}{2^n}+\log|a|.$$

Let $p_c(x)=x^2+c$. When $a\to 0$, H\'{e}non maps degenerate to a
$1$-dimensional map $x\mapsto p_c(x)$, acting on parabola $x=p_c(y)$.
When $a\to 0$, the Green's functions $G_{\lambda}^+$ converge to
$G_{(0,c)}^+(x,y)=G_{p_c}(x)$, where $G_{p_c}(x)$ is the Green's
function of the map $x\mapsto p_c(x)$. The functions
$G_{\lambda}^+$, $G_{\lambda}^-$ are pluriharmonic  on the escape
loci $U_{\lambda}^+$, $U_{\lambda}^-$ respectively. Therefore,
their level sets are foliated by Riemann surfaces. We denote by
${\cal F}_{\lambda}^+$, ${\cal F}_{\lambda}^-$ the corresponding
foliations. These Riemann surfaces are in fact copies of $\mathbb
C$ \cite{HOVI}.
% They have the following dynamical description:

There are also analogues  $\phi_{\lambda,+}$, $\phi_{\lambda,-}$
of the {\it B\"{o}ttcher coordinates}. The function
$\phi_{\lambda,+}$ is well defined and holomorphic in a
neighborhood $V_{\lambda}^+$  of $(x=\infty, y=0)$ in the
${\hat\C}^2 $-compactification of $\C^2$, and
$\phi_{\lambda,+}\sim x$ as $x\to \infty$. Moreover,  it
semiconjugates $f$ to $z\mapsto z^2$, $
\phi_{\lambda,+}(f_{\lambda})=\phi_{\lambda,+}^2$.

In $V^+_{\lambda}$, the foliation  ${\cal F}_{\lambda}^+$ consists
of the level sets of $\phi_{\lambda,+}$. It can be propagated to
the rest of $U_{\lambda}^+$ by the dynamics. One can also extend
$\phi_{\lambda,+}$ to $U_{\lambda}^+$  as a multi-valued function,
and then use any branch of it to define $\FF_{\lambda}^+$.
% In the domain where $\phi_{\lambda,+}$ is defined,
Moreover, any branch is related to the Green's function by $ G_{\lambda}^+=\log|\phi_{\lambda,+}|.$

The function $\phi_{\lambda,-}$ is defined in an analogous way.

\subsection{Critical Locus}

\begin{definition} The critical locus ${\cal C}_{\lambda}$ is the set of tangencies
between foliations ${\cal F}_{\lambda}^+$ and ${\cal F}_{\lambda}^-$.
\end{definition}

The critical locus is
%  non-empty for all H\'{e}non maps \cite{BS4}.
given by the zeroes of the 2-form
$$w=d\log\phi_{\lambda,+}\wedge d\log\phi_{\lambda,-}.$$
It is a non-empty proper analytic subset of $U_{\lambda}^+\cap U_{\lambda}^-$
which is invariant under the maps $f_{\lambda}$, $f_{\lambda}^{-1}$.

% The critical locus can be a good tool to browse in the parameter space of the H\'{e}non maps.
%The critical loci were described for H\'{e}non maps that are small perturbations of certain polynomials.
%\begin{conjecture} The critical locus is a quasiconformal invariant of
%a hyperbolic component of the H\'{e}non maps.
%\end{conjecture}

Lyubich and Robertson (\cite{LR}) gave a description of the
critical locus for H\'{e}non mappings $$(x,y)\mapsto
(p(x)-ay,x),$$ \noindent where $p(x)$ is a hyperbolic polynomial
with the connected Julia set, $a$ is sufficiently small. They showed
that for each critical point $c$ of $p$ there is a component the critical locus that is asymptotic to the line $y=c$. The rest of the components are iterates of these ones, and
each is a punctured disk. In this case,  all critical loci are obviously conformally equivalent.

A topological description of the critical locus for complex H\'{e}non maps that are perturbations
of quadratic polynomials with disconnected Julia sets is given in \cite{F}.
The critical locus is a connected Riemann surface with rich topology. It is composed of countably many Riemann spheres
$S_n$ with holes, that are connected to each other by handles. There are $2^{k-1}$ handles between $S_n$ and $S_{n+k}$. On each sphere $S_n$ the handles accumulate to two Cantor sets.

We are ready to formulate the main result of this paper:

\begin{theorem}\label{main2}
The critical loci of the H\'{e}non maps that are small
perturbations of quadratic polynomials with disconnected Julia
sets are quasiconformally equivalent.
\end{theorem}

\subsection{Topological description of the critical locus}
% First,  we recall relevant results from \cite{F}
In this section we will  give, following \cite{F},  a precise description of the critical
locus.

Let $\AAA$ be the space of one-sided sequences of $0$'s and $1$'s
(``infinite strings''),
and let $\AAA^n$ be the space of $n$-strings of $0$'s and $1$'s.

Let us describe {\it truncated spheres} that will serve as the building blocks for the
critical locus. Consider a 2-sphere $S\equiv S^2$ and a pair of disjoint
Cantor sets $\Sigma, \Theta \subset S$.
Let us  fix a nest of {\it figure-eight curves}
$\Gamma^n_\alpha$ and $L^n_\alpha$, $ n =  0,1,2,\dots$,  $\alpha\in
\AAA^{n}$,   respectively generating these Cantor sets in the
following natural way%
\footnote{For $n=0$, we let $\AAA^0=\emptyset$.}.

Let us start with a single figure-eight curve $\Gamma^0$
 bounding two domains $D^1_0$ and $D^1_1$
(with an arbitrary assignment of labeling).
The curve $\Gamma^1_0\subset D^1_0$  bounds two domains
$D^2_{00}$ and
$D^2_{01}$  compactly contained in $D^1_0$ (with an
arbitrary assignment of the second label) , and similarly,
$\Gamma^1_1\subset D^1_1$ bounds two domains $D^2_{10}$ and $D^2_{11}$
inside $D^1_1$, etc. See Figure \ref{fig:fig_eight}.

We assume that $\bigcup_\alpha D^n_\alpha \supset \Sigma$ and
 $\diam D^n_\alpha\to 0$ as $n\to \infty$ (uniformly in
$\alpha\in \AAA^{n}$), so for each sequence $\alpha\in \AAA$, there
is a unique point
$$
    \sigma_\alpha =\bigcap_{n=1}^\infty  \overline{ D^n_{\alpha_n} } \in \Sigma,
$$
where $\alpha_n\subset \AAA^n$ is the initial $n$-string of $\alpha$. That gives us a one-to-one coding
of points $\sigma\in \Sigma$ by sequences $\alpha\in \AAA$.

 Similarly, $\Theta$ is generated by a hierarchical nest of figure-eights
$L^n_\alpha$. We assume that these two nests are {\it disjoint}
in the sense that figure-eight $L^0$ lies in the unbounded
component of $\C\sm \Gamma^0$, and the other way around.

% Elements of $\Sigma$,
% $\Omega$ can be naturally encoded by one-sided infinite sequences, with entries
%$0$ and $1$.
% Denote by $\sigma_{\alpha}\in \Sigma$,
%$\omega_{\alpha}\in\Omega$ the elements encoded by a sequence
%$\alpha\in \AAA$.

The singular points $\sigma^n_\alpha$ and $\theta^n_\alpha$
of the figure-eights $\Gamma^n_\alpha$ and $L^n_\alpha$ respectively are called
their  {\it centers}.
% Let $\alpha_n$ be a $n$-string with elements $0$ and $1$.
For each figure-eight $\Gamma^n_\alpha $,
select a disk $V^n_\alpha\ni \si^n_\alpha$  whose closure  is disjoint from all
other figure-eights $\Gamma^m_\beta$ and from  $L^0$.
Then  select a disk $U^n_\alpha\ni \theta^n_\alpha$ with similar properties  for each
figure-eight $L^n_\alpha$. Moreover, make these choices so that  the closures of
all these disks are pairwise disjoint.

%$V_{\alpha_n}\subset S\backslash\left(\Sigma\cup \Omega\right)$
%with the smooth boundary.  We require that the $\bar V_{\alpha_n}$
%are disjoint and  converge to
%$\sigma_{\alpha}$, where $\alpha_n$ is the string of the first $n$
%elements of $\alpha$. We assume that sequence of the length $0$ is parametrized by $\emptyset$.

\begin{figure}[h]
\centering \psfrag{sigma}{$\Sigma$} \psfrag{Vem}{$V^{0}$}
\psfrag{V0}{$V^1_0$}\psfrag{V1}{$V^1_1$}\psfrag{G0}{$\Gamma^0$}\psfrag{G10}{$\Gamma^1_0$}\psfrag{G11}{$\Gamma_1^1$}
\includegraphics[height=3.5cm]{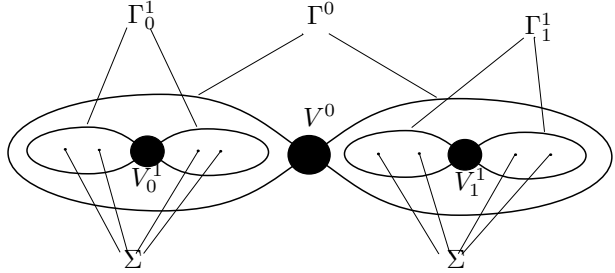}
\caption{The geometry of a truncated sphere}\label{fig:fig_eight}
\end{figure}

%Disks $U_{\alpha_n}$ play the same role for $\Omega$ as $V_{\alpha_n}$
%for $\Sigma$, and  we assume that $\bar  U{\alpha_n} \cap V_{\alpha_m} =\emptyset$.
For each $n\in \N$,
 $\alpha \in \AAA^{n}$, we choose a homeomorphism $h_{\alpha}^n$ between
the boundaries of $V_{\alpha}^n$ and $U_{\alpha}^n$.
Finally, we  mark  a point  $p\in S$ in the exterior of both figure-eights and
the disks $\bar U^0$, $\bar V^0$.
With all these choices in hand, we call
$$
      S\sm X, \quad  \mathrm{where}\  X:=
      \Sigma\cup\Theta\cup \{p\} \bigcup_n \left(\bigcup_{\alpha\in{\cal A}^n} U_{\alpha}^n
    \cup V_{\alpha}^n \right),
$$
 a {\it truncated sphere}.
 Note that for any two truncated spheres $S\sm  X$ and $S'\sm X'$ there is a homeomorphism $(S,X)\ra (S', X')$ that restricts to
the natural homeomorphisms between the corresponding marked sets.

\begin{theorem}\label{theorem_main}
Assume that the quadratic polynomial  $x\mapsto x^2+c$ has disconnected
Julia set. Then there
exists $\delta>0 $ such that for any $|a|<\delta$ the critical locus
of the H\'enon  map
$$
f_{\lambda}: \left(\begin{array}{c}x \\y \end{array}\right)\mapsto
\left(\begin{array}{c}x^2+c-ay \\ x \end{array}\right)
$$
is a non-singular Riemann surface that admits the following topological model. Take countably
many copies $S_m \sm X_m $, $m \in \mathbb Z$, of the truncated
sphere $S\sm X$,  and glue the boundary of  $V_{\alpha}^n$ of $S_k$ to the boundary of
 $U_{\alpha}^n $ of $S_{n+k+1}$ by means of the homeomorphism
$h_{\alpha}^n$.
The model map acts by translating $S_n\sm X_n$
to $S_{n+1}\sm X_{n+1}$.
\end{theorem}

\begin{figure}[h]
\centering
\includegraphics[height=11cm]{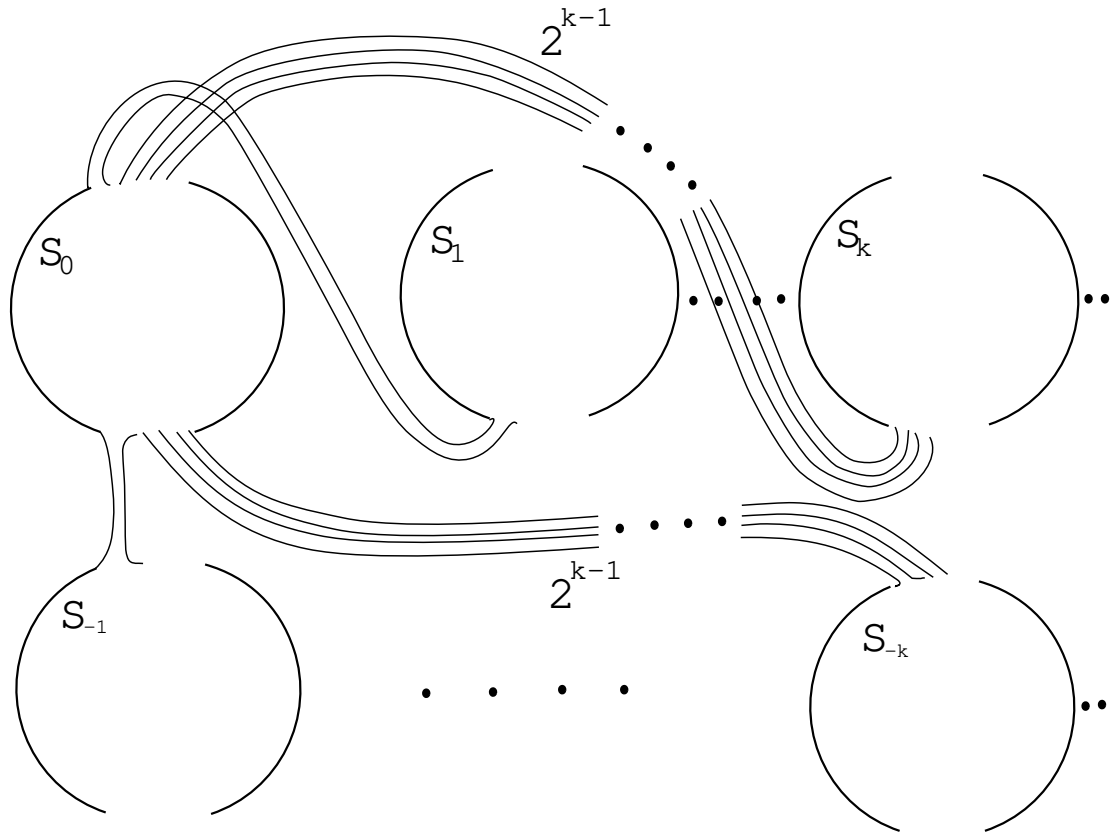}
\caption{Critical Locus.}\label{fig:top_model}
\end{figure}

\subsection{Proof of Theorem \ref{main2}}

In \cite{F} we gave a detailed description of the position of the
critical locus ${\cal C}_{\lambda}$ in $\mathbb C^2$ for
$\lambda\in \Lambda$, where $\Lambda$ is a set of parameters of a
small perturbation of quadratic polynomials with disconnected
Julia set. Below we fix a parameter $\lambda_0=(a_0,c_0)\in
\Lambda$ and use the description from \cite{F} to construct a
holomorphic motion of the critical loci $\cal C_{\lambda}$, for
$\lambda$ that belong to $1$-parameter family in a neighborhood of
$\lambda_0$.
% In order to apply Theorem \ref{main},
Let us first describe a fundamental domain of the critical locus
in $\mathbb C^2$.
\begin{figure}[h!]
\centering \psfrag{omega}{$\Omega_{\lambda}$}\psfrag{x}{$x$}\psfrag{y}{$y$}
\psfrag{1Omeg}{$\Upsilon_{\lambda}$} \psfrag{x=infty}{$x=\infty$}
\psfrag{|a|alpha}{$|a|\alpha$}
\includegraphics[height=5cm]{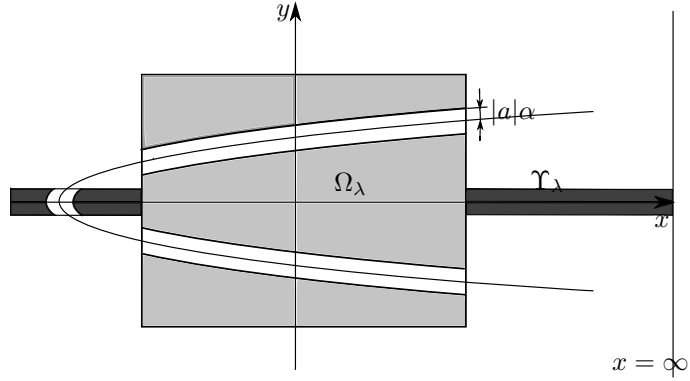}
\caption{Domains $\Omega_{\lambda}$ and $\Upsilon_{\lambda}$}
\end{figure}
Let
$$
  \Omega_{\lambda} = \{ ( x,y)\in \mathbb C^2:\quad   G_a^+\leq r, \, \quad
  |y|\leq \alpha, \, \quad |p_c(y)-x|>
  |a|\alpha\}
$$
$$
 \Upsilon_{\lambda}=\{ (x,y)\in \mathbb C^2 :\  G_a^+(x,y)\geq r , \quad
|y| \leq \epsilon\}.
$$

When $a\to 0$, domains $\Omega_{\lambda}$ converge in Hausdorff topology to $\Omega_{(0, c)}$.
In \cite{F} we choose $r$, $\alpha$ and $\epsilon$, depending on $c$, so that for $c'$ close to $c$ and $a$ small enough,
${\cal C}_{\lambda}\cap (\Omega_{\lambda}\cup \Ups_{\lambda})$ form
a fundamental domain for the map $f_{\lambda}$ on the critical locus.
We further  cut $\Omega_{\lambda}\cap U_{\lambda}^+$ into subdomains
$\Omega_{\lambda}^{\alpha}$, where $\alpha$ goes over all finite diadic strings.

We recursively encode the $n$-th preimages $\xi_{ \alpha} $
of $0$ under the map $z\mapsto z^2+c$ by diadic $n$-strings $\alpha$. We assume that $0$ itself is parametrized by $\emptyset$. Let $\alpha^0$, $\alpha^1\in \AAA^{n+1}$ be the
strings obtained by adding $0$, $1$ correspondingly to $\alpha$ on the right. We encode preimages of $\xi_{\alpha}$
by $\alpha^0$ and $\alpha^1$.
% Let $\alpha_n$ be a duadic $n$-string, then we denote by $\xi_{\alpha_n}$ the corresponding point.
% It is shown in \cite{F} that the
Since each  connected component of
$$\{\frac{r}{2^{n+1}}\leq G_{p_c}\leq \frac{r}{2^n}\}$$
 contains a unique  $n$-preimage of the critical point, they are
 encoded by diadic $n$-strings as well.
$$
\Omega_{(0,c)}^{\alpha}=\{\mbox{a connected component of }\{\frac{r}{2^{n+1}}\leq G_{(0,c)}^+\leq \frac{r}{2^n}\}\cap
\Omega_{(0,c)}\}
$$
that contains a line $x=\xi_{\alpha}$, $\alpha\in \AAA^{n}$.

By the choice of $r$ in \cite{F}, the connected components of
$$
\{\frac{r}{2^{n+1}}\leq G_{\lambda}^+\leq \frac{r}{2^n}\}\cap
\Omega_{\lambda}
$$
depend continuously on $a$ in the Hausdorff topology.
We denote by $\Omega_{\lambda}^{\alpha}$ continuation of $\Omega_{(0,c)}^{\alpha}$.

\begin{figure}[h!]
\centering \psfrag{|y|=alpha}{$|y|=\alpha$}
\psfrag{G=r}{$G_a^+=\frac{r}{2^n}$}\psfrag{G=r/2}{$G_a^+=\frac{r}{2^{n+1}}$}
\psfrag{u=aalpha}{$|u_c|=a\alpha$}\psfrag{x}{$x$}\psfrag{y}{$y$}
\includegraphics[height=7cm]{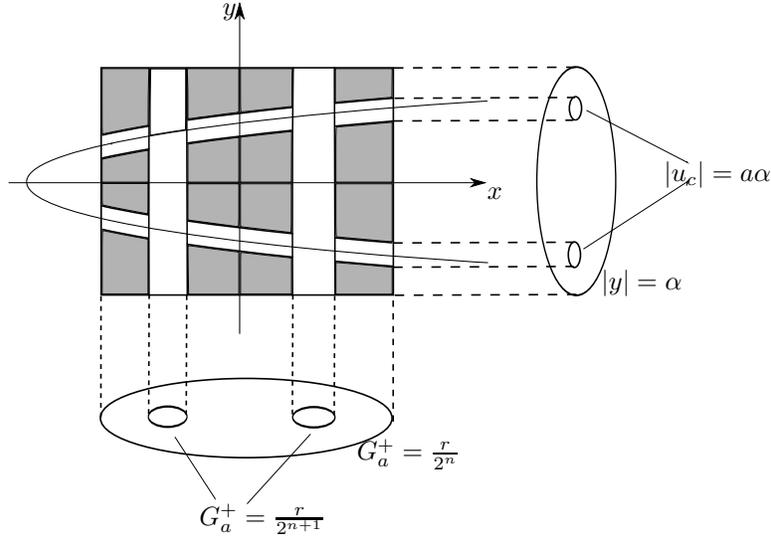}
\caption{Domain $\Omega_a^{\alpha}$}
\end{figure}
% \begin{proof}[Proof of the Theorem \ref{main2}]

Let $u_c=y^2+c-x$.

\begin{lemma}[{\cite[Lemma 11.4]{F}}] In $\Omega_{\lambda}^{\alpha}$,
where $\alpha\in \AAA^n$, $n=0,1,\dots$, the critical locus is
a connected sum of two disks $D_1$ and $D_2$ with two holes each. The
boundary of $D_1$ belongs to $\{|y|=\alpha\}$, and the holes of
$D_1$ have boundaries on $\{|u_c|=|a|\alpha\}$. The boundary of
$D_2$ belongs to $\{G_{\lambda}^+=\frac{r}{2^n}\}$ and the holes to
$\{G_{\lambda}^+=\frac{r}{2^{n+1}}\}$.
\end{lemma}

A holomorphic motion near the boundaries
$\{G_{\lambda}^+= 2^{-n}  r \}$ and $\{G_{\lambda}^+~=~2^{-(n+1)}~r~\}$ is defined so
that it preserves the values of
the  functions $\phi_{\lambda,+}^{2^n}$ and
$\phi_{\lambda,+}^{2^{n+1}}$ respectively. Similarly, the holomorphic
motion of the boundaries $\{|y|~=~\alpha\}$
and $\{|u_c|=|a|\alpha\}$ preserves the values of $y$ and $u_c$.

We apply Theorem \ref{main} to the piece of the critical locus inside $\Omega_{\lambda}^{\alpha}$ and extend the holomorphic
motion to the interior.

\begin{lemma}[{\cite[Lemma 13.1]{F}}] There exists $\delta$ such that
$\forall |a|<\delta$ the critical locus ${\cal C}_{\lambda}$ in
$\Ups_{\lambda}$ is a punctured disk, with a hole removed. The
puncture is at the point $(\infty, 0)$, the boundary of the hole
belongs to $\{|p_c(y)-x|=|a|\alpha\}$.
\end{lemma}

We apply Theorem \ref{main} to ${\cal C}_{\lambda}\cap \Ups_{\lambda}$ and extend the holomorphic motion to the interior.
We propogate the holomorphic to the rest of ${\cal C}_{\lambda}$ by dynamics. The space $\Lambda$ is path connected. Therefore, the critical loci ${\cal C}_{\lambda}$ for all maps that are small perturbations of quadratic polynomials with disconnected Julia set are quasiconformally equivalent.

% \end{proof}


\begin{bibdiv}
\begin{biblist}
\bib{AS}{book}{
author={Ahlfors, L.},
author={Sario, L.},
title={Riemann surfaces},
date={1960},
publisher={Princeton University Press}
}

\bib{BB}{article}{
author={Bassanelli, G.},
author={Berteloot, F.},
title={Lyapunov Exponents, bifurcation currents and laminations in bifurcation locus},
journal={Math.Ann.},
volume={345},
date={2009},
pages={1-23}
}
\bib{BS4}{article}{
title={Polynomial
diffeomorphism of $\mathbb C^2$. V. Critical points and Lyapunov
exponents},
author={Bedford, E.},
author={Smillie, J.},
journal={J. Geom. Anal.},
volume={8(3)},
pages={349-383},
date={1998}}

\bib{BS5}{article}{
title={Polynomial diffeomorphisms of $\mathbb C^2.$ VI. Connectivity of J},
author={Bedford, E.},
author={Smillie, J.},
journal={Ann. of Math.},
volume={148(2)},
pages={695-735},
date={1998}}

\bib{BR}{article}{
author={Bers, L.},
author={Royden, H.},
title={Holomorphic families of injections},
journal={Acta Math.},
volume={157},
date={1986},
pages={259-286}
}

\bib{F}{article}{
author={Firsova, T.}, title={Critical locus for complex H\'{e}non
maps}, journal={Indiana Univ. Math. J.},
volume={61},date={2012},pages={1603-1641}}

\bib{DL}{article}{
author={Dujardin, R.},
author={Lyubich, M.},
title={Stability and bifurcations for dissipative polynomial automorphisms of $\C^2$},
journal={Preprint IMS at Stony Brook},
volume={13-01},
date={2013},
}

\bib{HOVI}{article}{
title={H\'{e}non mappings
in the complex domain. I. The global topology of dynamical space},
author={Hubbard, J. H.},
author={Oberste-Vorth, R. W.},
journal={Inst. Hautes \'{E}tudes Sci. Publ. Math.},
volume={79},
pages={5-46},
date={1994}}

\bib{L}{article}{
title={Some typical properties of the dynamics of rational maps},
author={Lyubich, M.},
journal= {Russian Math. Surveys},
date={1983},
volume={38},
pages={154-155}
}

\bib{LR}{article}{
title={The critical
locus and rigidity of foliations of complex H\'{e}non maps},
author={Lyubich, M.},
author={Robertson, J.},
journal={Unpublished manuscript},
date={2004}}

\bib{Teich}{book}{
author={Hubbard, J. H.},
title={Teichm\"{u}ller Theory and
Applications to Geometry, Topology and Dynamics},
volume={1},
date={2006},
publisher={Matrix Editions} }

\bib{MSS}{article}{
author={Ma\~{n}\'{e}, R.},
author={Sad, P.},
author={Sullivan, D.},
title={On the dynamics of rational maps},
journal={Ann. Sci. \'{E}cole Norm. Sup.},
volume={16},
date={1983},
pages={193-217}
}

%\bib{Rudin}{article}{
%author={Rudin, W.},
%title={The closed ideals in an algebra of analytic functions.},
%journal={Canad. J. Math.},
%volume={9},
%date={1957},
%pages={426-434}}

\bib{Slodkowski}{article}{
author={Slodkowski, Z.},
title={Holomorphic motions and polynomial hulls},
journal={Proc. Amer. Math Society},
volume={111},
date={1991},
pages={347-355}}

\bib{ST}{article}{
author={Sullivan, D.},
author={Thurston, W.},
title={Extending holomorphic motions},
journal={Acta Math},
volume={157},
date={1986},
pages={243-257}
}

\bib{Voichick}{article}{
author={Voichick, M.},
title={Ideals and invariant subspaces of analytic functions},
journal={Transactions of Am. Math. Soc.},
volume={111},
date={1964},
pages={493-512}
}

\end{biblist}
\end{bibdiv}
\end{document}